%% file: DS.tex
\pdfoutput=1 
\RequirePackage[l2tabu, orthodox]{nag} 

\documentclass[12pt, reqno]{amsart}

\usepackage{geometry}

\usepackage{amssymb, amsmath, amsthm}

\usepackage[all,warning]{onlyamsmath} 

\usepackage[dvipsnames]{xcolor}

\usepackage[backref, colorlinks=true, citecolor=Green]{hyperref}
\usepackage[alphabetic,backrefs,lite]{amsrefs}
\usepackage{amscd}   
\usepackage{fullpage}
\usepackage{colonequals}
\usepackage{enumerate}
\usepackage{tikz-cd} 
\usepackage[all,cmtip]{xy}
\usepackage{enumitem}

\usepackage{bm}
\usepackage{comment}

\usepackage{amsrefs}

\newtheorem{lemma}{Lemma}[section]
\newtheorem{theorem}[lemma]{Theorem}

\newtheorem{conj}[lemma]{Conjecture}

\newtheorem{claim*}{Claim}

\theoremstyle{remark}
\newtheorem{remark}[lemma]{Remark}

\newcommand{\HH}{{\mathbb H}}  
\newcommand{\PP}{{\mathbb P}}
\newcommand{\C}{{\mathbb C}}
\newcommand{\F}{{\mathbb F}}
\newcommand{\Q}{{\mathbb Q}}

\newcommand{\Z}{{\mathbb Z}}
\newcommand{\M}{\mathbb{M}}
\newcommand{\B}{\mathbb{B}}
\renewcommand{\S}{\mathbb{S}}
\newcommand{\T}{{\mathbb{T}}}   

\newcommand{\A}{{\mathcal A}}
\newcommand{\cuspidal}{\mathcal{C}}

\DeclareMathOperator{\Frob}{Frob}

\DeclareMathOperator{\im}{im}
\DeclareMathOperator{\End}{End}

\DeclareMathOperator{\Gal}{Gal}

\DeclareMathOperator{\Div}{Div}

\DeclareMathOperator{\tors}{tors}
\DeclareMathOperator{\Ver}{Ver}

\DeclareMathOperator{\Princ}{Princ}

\DeclareMathOperator{\SL}{SL}

\DeclareMathOperator{\hol}{hol}

\newcommand{\bs}{\backslash} 

\newcommand{\diamondop}[1]{\langle #1 \rangle} 
\renewcommand{\Im}{\operatorname{Im}}

\newcommand{\mtwo}[4]   {\left[\begin{matrix}\hfill{}#1&\hfill{}#2\\\hfill{}#3&\hfill{}#4\end{matrix}\right]}  
\newcommand{\smat}[4]{{\left[\begin{smallmatrix} #1 & #2 \\ #3 & #4 \end{smallmatrix}\right]}} 

\numberwithin{equation}{section}
\numberwithin{table}{section}

\newcommand{\defi}[1]{\textsf{#1}} 

\newcommand{\cd}{\cdot\allowbreak}
\usepackage{array}    
\usepackage{longtable}
\usepackage{booktabs}
\usepackage{caption}
\title{On Some Open Cases \\ of a Conjecture of Conrad, Edixhoven and Stein}

\author{Davide De Leo}
\address{Dipartimento di Matematica e Informatica, Università della Calabria, Ponte Pietro Bucci, Cubo 31B, 87036 Arcavacata di Rende (CS), Italy}
\email{davide.deleo@unical.it}

\author{Michael Stoll}
\address{Mathematisches Institut, Universit\"at Bayreuth, 95440 Bayreuth, Germany}
\email{Michael.Stoll@uni-bayreuth.de}
\urladdr{\url{https://www.mathe2.uni-bayreuth.de/stoll/}}

\date{\today}
\keywords{modular Jacobian variety, rational torsion subgroup, cuspidal divisor class group \\ The first author is a member of INdAM--GNSAGA}

\begin{document}

    \begin{abstract}
Let \( p \geq 5 \) be a prime.~In 2003 Conrad, Edixhoven, and Stein conjectured that the rational torsion subgroup of the modular Jacobian \( J_1(p) \) coincides with the rational cuspidal divisor class group. Using explicit computations in Magma, the open case \( p = 29 \) has been proven by Derickx, Kamienny, Stein, and Stoll in 2023. We extend these results to primes \( p = 97, 101, 109, \) and \( 113 \). In addition, we provide a list of the groups \( J_1(p)(\mathbb{Q})_{\text{tors}} \) for every prime up to \( p \leq 113 \). However, our method is general and can be applied to larger primes.
\end{abstract}

	\maketitle

\section{Introduction}
\noindent Let \( k \) be an algebraic number field and let \(\A\) be an abelian variety over \( k \). By the Mordell--Weil Theorem, the set \(\A(k)\) of \( k \)-rational points of \(\A\) is a finitely generated abelian group:
\[
\A(k) = \Z^r \oplus \A(k)_{\tors},
\]
where \( r \) is a nonnegative integer called the (arithmetic) \emph{rank} of \(\A\), and \(\A(k)_{\tors}\) is a finite group, known as the \( k \)-\emph{rational torsion subgroup} of \(\A\).

In general, determining the rank of a given abelian variety is a challenging problem (even in small dimensions). In contrast, in this paper we address the more ``affordable'' problem of finding, and possibly computing, the \(\Q\)-rational (or simply rational) torsion subgroup \(\A(\Q)_{\tors}\) for certain abelian varieties.

The simplest examples of abelian varieties that one can readily construct are the Jacobian varieties of curves. Let \( X \) be a smooth projective algebraic curve of genus \( g \geq 1\) defined over a field \( k \), and assume that $X$ admits at least one $k$-rational point. Recall that the Jacobian variety \( J(X) \) of \( X \) is an abelian variety over \( k \) of dimension \( g \); that is, \( J(X) \) is a commutative projective algebraic group of dimension \( g \) whose underlying abstract group is given by
\[
J(X)(k) \cong \Div^0(X) / \Princ(X),
\]
where \(\Div^0(X)\) denotes the group of degree-zero divisors on \( X \) and \(\Princ(X)\) is the group of principal divisors. By fixing a \( k \)-rational point \( Q \) on \( X \), one can define the Abel--Jacobi map
\[
\iota_Q : X(k) \rightarrow J(X)(k), \quad P \longmapsto [P-Q],
\]
which is injective and realizes \( X \) as a subvariety of \( J(X) \). This provides a useful  embedding that allows us to deduce properties of \( X \) by studying its corresponding Jacobian, which typically exhibits a far richer structure.

In the case \( g = 1 \), we have \( X = E \) is an elliptic curve and, consequently, \(\A = J(X) = E\). Over \(\Q\), the problem of classifying the rational torsion subgroup has been completely resolved. In his celebrated 1977 paper~\cite{mazur:eisenstein}, Mazur completely classified (up to isomorphism) all possible structures of the \(\Q\)-rational torsion subgroup in the case of elliptic curves, and many algorithms (e.g.~Nagell--Lutz theorem~\cite{Sil:AEC}*{VIII.7.2}) exist for computing the torsion subgroup in practice. Over a general number field \( k \), however, the problem is more subtle and remains an active area of research; a vast literature on the subject is available; see, e.g., \cite{DKKS} and the references given there.

For $ g > 1$, as one might expect, the problem becomes considerably more difficult. Therefore, as indicated earlier, we restrict our attention to the rational torsion subgroup of Jacobians of certain specific algebraic curves known as modular curves, which will be discussed in greater detail in Section~\ref{sec:preliminaries}.

Let $\Gamma\subseteq\SL_2(\Z)$ be a congruence subgroup, and consider the modular curve $X(\Gamma)$ and its associated (modular) Jacobian variety $J(\Gamma)=J(X(\Gamma))$. 

On \( X(\Gamma) \), we can consider the free abelian group \( \tilde\cuspidal(X(\Gamma))\) of degree-zero divisors supported on the cusps. Its image in the Jacobian under the Abel--Jacobi map is called the \emph{cuspidal} $($\emph{divisor class}$)$ \emph{subgroup} of $J(\Gamma)$. This subgroup plays a special role in the study of the torsion of modular Jacobians. It was shown by Manin~\cites{manin:parabolic} and Drinfeld~\cites{drinfeld} that its order is finite. If \(\cuspidal(X(\Gamma))^{\Gal_{\Q}}\) denotes the \emph{rational cuspidal subgroup}, i.e.~the subgroup of \(\cuspidal(X(\Gamma))\) that is invariant under the action of the absolute Galois group over $\Q$, we have the inclusion:
\begin{equation} \label{eq:cusp-jac_inclusion}
    \cuspidal(X(\Gamma))^{\Gal_{\Q}}\subseteq J(\Gamma)(\Q)_{\tors}.
\end{equation}

In order to further investigate the problem, we first need to better understand the structure of the rational cuspidal subgroup. Naively, one might consider the subgroup of \(\tilde\cuspidal(X(\Gamma))\) generated by differences of rational cusps and then take its image in the Jacobian, denoted by \(\cuspidal(X(\Gamma))^\Q\). However, it can happen that differences of non-rational cusps yield rational divisor classes. Thus, a priori, this only provides the inclusion
\begin{equation}    \label{equation:diff_cusps-rat_cusps_inclusion}
\cuspidal(X(\Gamma))^\Q \subseteq \cuspidal(X(\Gamma))^{\Gal_{\Q}}.  
\end{equation}
Nevertheless, for prime level it is likely that this inclusion is an equality, since the non-rational cusps typically form one rather large Galois orbit so that, in the subgroup fixed by the Galois action, they always have to show up together. This is indeed the case for \(\Gamma = \Gamma_0(p)\), and we will show that it also holds for \(\Gamma = \Gamma_1(p)\) in some cases.

\begin{remark}
If the congruence subgroup $\Gamma$ is taken too generally and no additional information is given, then the inclusion in~\eqref{equation:diff_cusps-rat_cusps_inclusion} can fail to be an equality. To see this, we present the following example taken from the~\href{https://www.lmfdb.org/}{LMFDB}. Consider the modular curve with label~\href{https://beta.lmfdb.org/ModularCurve/Q/7.168.3.a.1/}{7.168.3.a.1}, which we denote by $X$. It is a genus 3 plane quartic with no rational cusps (hence $\cuspidal(X)^{\Q}$ is trivial), but it does have two Galois orbits of length three. The curve $X$ has $\Q$-gonality $\gamma$ between 3 and 4, so it is not hyperelliptic and one checks via the canonical model on the LFMDB that $X$ has no rational points (since no $\Q_p$-points for $p=2$). By the Riemann--Roch theorem, it follows that such a curve cannot have $\Q$-gonality 3 (since otherwise it should have a rational point), so that $\gamma$ is actually 4. This implies that the difference of the two orbits of length 3 cannot vanish on the Jacobian (otherwise one would obtain a rational function on the curve with exactly 3 zeros and 3 poles, which gives a $\Q$-rational morphism $X\rightarrow\PP^1$ of degree 3) and gives a nontrivial element in $\cuspidal(X)^{\Gal_\Q}$.

\end{remark}

\noindent One may wonder whether \eqref{eq:cusp-jac_inclusion} is an equality: 
\begin{itemize}
\item For prime level $p$ and $\Gamma=\Gamma_0(p)$, Ogg first determined the structure of the rational cuspidal subgroup in~\cite{ogg1} and then conjectured its equality with the whole rational torsion of $J_0(p)$ in~\cite{ogg2}. This result was later proved by Mazur~\cite{mazur:eisenstein}. For composite level $N$, the problem remains open and is known as the \defi{Generalized Ogg's Conjecture}. Nevertheless, some results under additional assumptions on $N$ have been obtained (see~\cites{lorenzini,ohta2,ribet:wake,Yoo1,Yoo2} and the references therein for additional information).
\item For prime level $p$ and $\Gamma=\Gamma_1(p)$, in 2003 Conrad, Edixhoven and Stein presented the following~(\cite{CES}*{6.2.2.}) 
\end{itemize}
\begin{conj}    \label{conj:ces}
Let $p\geq5$ be a prime. The $\Q$-rational torsion subgroup of \(J_1(p)\) is generated by differences of $\Q$-rational cusps on $X_1(p)$.
\end{conj}
In their work, the authors provided strong numerical evidence proving Conjecture~\ref{conj:ces} for every prime \(p \leq 157\), except for the cases $p=29, 97,101,109$ and $113$. Later, in~\cite{Ohta1} Ohta proved a significant part of the conjecture, showing its validity up to the $2$-torsion. However, both approaches verify the conjecture only indirectly, without actually computing the groups involved. Consequently, although they yield positive results in some cases, the answer remains somewhat incomplete, since no explicit information about the subgroup structure is provided.

A more constructive approach to the conjecture can be found in~\cite{DKKS}*{Theorem~3.2}. Some of the computational methods employed in that work enabled the authors to describe the explicit structure of the rational torsion subgroup of $J_1(29)(\Q)_{\tors}$, thereby confirming Conjecture~\ref{conj:ces} for the smallest open case. 

In this direction, we extend these techniques to the remaining smallest unresolved primes. More precisely, we prove the following result:
\begin{theorem} \label{theorem:CESholds}
We have
\begin{equation}    \label{equation:ces}
    \cuspidal_1(p)^{\Q}:=\cuspidal(X_1(p))^{\Q}=J_1(p)(\Q)_{\tors},\quad\text{for }p=97,101,109,113.
\end{equation}
\end{theorem}
\noindent which, by~\eqref{eq:cusp-jac_inclusion}, confirms both Conjecture~\ref{conj:ces} and the equality in~\eqref{equation:diff_cusps-rat_cusps_inclusion} for those cases.  

In Section~\ref{sec:strategy} we present the overall strategy and the proof of Theorem~\ref{theorem:CESholds}, which relies on the use of certain linear operators, derived from the Hecke operators in the theory of modular forms, that annihilate the rational torsion subgroup of the modular Jacobian. At the end of the section, we also provide tables summarizing the computational results (see Table~\ref{subsec:tableproof}).

Furthermore, by applying the same methods to explicitly determine the structure of \(J_1(p)(\Q)_{\tors}\), we also give another proof of both Conjectures~\ref{conj:ces} and the equality in~\eqref{equation:diff_cusps-rat_cusps_inclusion} for the primes previously addressed in~\cite{CES} up to \(113\). The main contribution in this case is the precise description of the rational torsion subgroups (see Table~\ref{subsec:tablelist1}).  

For composite level \(N\), the problem of determining whether the rational torsion subgroup of $X_1(N)$ is equal to its rational cuspidal subgroup has been studied, and a generalization of the conjecture of Conrad, Edixhoven and Stein has been proposed (see~\cite{DerEtal}*{Conj.~4.15}, and the discussion in Section~4.6 of the same paper for further details).

\subsection{Comments on the code}
This paper includes a significant computational component, utilizing the computer algebra system Magma~\cite{magma} to carry out the necessary calculations. The computations were performed using Magma V2.28-3 on a Linux machine equipped with 16 GB of RAM (btm2x4) at the University of Bayreuth, as well as on the first author's personal laptop. The code supporting our claims is available at~\cite{code}. 

\section*{Acknowledgements}
This work is based on the master's thesis of the first author~\cite{DeLeo}, partially conducted during an Erasmus exchange at the University of Bayreuth in Germany under the supervision of the second author. The first author would like to express sincere gratitude to the University of Bayreuth for its warm welcome and hospitality, and especially to the computational team for providing access to the computational resources that were essential for carrying out the calculations presented in this article. Special thanks also go to Laura Paladino for her early review of both the thesis and this manuscript, as well as for her valuable suggestions. The authors further thank Maarten Derickx for his helpful comments and for patiently clarifying certain issues arising from the misformulation of a conjecture in an earlier version of this draft.

\section{Preliminaries} \label{sec:preliminaries}

\subsection{Review of modular curves}
In the present section, we give a brief overview on the modular curves $X(\Gamma)$
associated to some subgroups $\Gamma$ of finite index of $\SL_2(\Z)$. We refer the reader to~\cite{shimura} and~\cite{diamond-im} for further details.

Let \(\HH =\{z\in\C : \Im(z)>0\}\) denote the complex open upper half plane. The group $\SL_2(\Z)$ acts both on $\HH$ and on the extended upper half plane $\HH^{*}=\HH\cup\PP^{1}(\Q)$ via M\"obius transformations
\[\gamma z=\frac{az+b}{cz+d},\qquad
\text{where }z\in\HH^* \text{ and }\gamma = \mtwo{a}{b}{c}{d}\in\SL_2(\Z).\] 

\noindent We fix a positive integer $N\geq 1$, and denote by $\Gamma(N)$ the kernel of the reduction modulo $N$ map, that is the set \(\{\gamma \in \SL_2(\Z): \gamma \equiv 1 \mod{N} \}\). Recall that $\Gamma\subseteq\SL_2(\Z)$ is said to be a \emph{congruence subgroup} of $\SL_2(\Z)$ if and only if $\Gamma(N)\subseteq\Gamma$ for some $N$, and the minimum of such $N$'s is called the \emph{level} of the subgroup. We observe that every congruence subgroup has finite index in $\SL_2(\Z)$~\cite{shimura}*{Section~1.6}, but there are subgroups of finite index which are not congruence subgroups~\cite{birch}.  

Given a congruence subgroup $\Gamma$, the action of $\SL_2(\Z)$ induces an action of $\Gamma$ on $\HH$ and on $\HH^{*}$. The quotient space $Y(\Gamma)(\C) = \Gamma\bs\HH$ is a noncompact open Riemann surface. Since $\Gamma$ has finite index in $\SL_2(\Z)$, the set $S_{\Gamma}=\Gamma\bs\PP^{1}(\Q)$ contains only finitely many points, called \emph{cusps}. Adding these points to $Y(\Gamma)(\C)$, we obtain the compact Riemann surface $X(\Gamma)(\C)=\Gamma\bs\HH^{*}$. The spaces $Y(\Gamma)$ and $X(\Gamma)$ are also called \emph{modular curves}. 

\noindent We restrict our attention to the following congruence subgroups of $\SL_2(\Z)$:
\begin{align*}
    \Gamma_0(N) =& \left\{\mtwo{a}{b}{c}{d} \in \SL_2(\Z) : N \mid c \right\}, \\
    \Gamma_1(N) =& \left\{\mtwo{a}{b}{c}{d} \in \SL_2(\Z) : (c,d) \equiv (0,1) \bmod{N} \right\} 
\end{align*}
and we write 
\begin{align*}
    Y_0(N)(\C) =& Y(\Gamma_{0}(N))(\C), & X_0(N)(\C) =& X(\Gamma_{0}(N))(\C) \\
    Y_1(N)(\C) =& Y(\Gamma_{1}(N))(\C), & X_1(N)(\C) =& X(\Gamma_{1}(N))(\C)
\end{align*}
for the respective modular curves and the corresponding compactifications.

\smallskip

\noindent \textbf{Moduli Space Setting.} 
Modular curves parametrize elliptic curves equipped with extra torsion data. More precisely, 
\begin{itemize}
\item the complex points of $Y_0(N)$ are in natural bijection with the isomorphism classes of pairs $(E,C)$, where $E$ is an elliptic curve over $\C$ and $C$ is a cyclic subgroup of $E(\C)$ of order~$N$.
\item the complex points of $Y_1(N)$ are in natural bijection with the isomorphism classes of pairs $(E,P)$, where $E$ is an elliptic curve over $\C$ and $P$ is a point on $E(\C)$ of exact order~$N$. 
\end{itemize}
The modular curve \(X_0(N)(\C)\)  has a canonical model over \(\Q\). We denote it by \(X_0(N)_{\Q}\) (or simply by \(X_0(N)\)). About the modular curve \(X_1(N)(\C)\), we have seen that its noncuspidal points parametrize pairs \((E,P)\), where \(E\) is a complex elliptic curve and \(P\in E(\C)\) is a point of exact order \(N\). There is another (more natural) moduli interpretation for the points of \(Y_1(N)(\Q)\) as pairs \((E,\phi)\), where \(E\) is an elliptic curve over \(\Q\) and \(\phi\colon \mu_N\to E\) is an embedding of group schemes of the group \(\mu_N\) of the \(N\)th roots of unity into \(E\). 

These two moduli interpretations are different but yield $\Q$‑models for the compactified modular curve $X_1(N)$, which are isomorphic over~$\Q$ as follows: consider a point $(E,\phi)$ and a choice of primitive root $\zeta\in\mu_N$. If $e_N$ denotes the Weil pairing, then the set of $Q\in E[N]$ such that $e_N(Q,\phi(\zeta))=\zeta$ is a coset $C$ of the subgroup $\im(\phi)\subset E[N]$.  We then set $E' = E/\im(\phi)$ and $P = C \bmod\im(\phi)$. We will consider this latter model, which we denote by \(X_1(N)_{\Q}\) (or simply by \(X_1(N)\)). We remark that in our choice, the cusp at infinity is $\Q$-rational. For more information about models of modular curves see~\cite{shimura}*{Chapters 1,6} and~\cite{diamond-im}*{II.~8}.

There is a natural map \(X_{1}(N)\rightarrow X_{0}(N)\) (induced over \(\C\) by the identity on \(\HH^{*}\)).~This makes \(X_{1}(N)\) into a (possibly ramified) Galois covering of \(X_{0}(N)\), whose Galois group consists of the diamond operators \(\langle d \rangle\ \text{for}\ d\in(\Z/N\Z)^{\times}, \text{ where } \langle-1\rangle\) is the identity, so the Galois group is naturally isomorphic to \((\Z/N\Z)^{\times}/\{\pm 1\}\). In terms of the interpretation of points on \(Y_{1}(N)\) as pairs \((E,\phi\colon\mu_{N}\rightarrow E)\), the action of \(\langle d \rangle\) corresponds to pre-composing \(\phi\) with the \(d\)th power map. If \(H\subseteq (\Z/N\Z)^{\times}/\{\pm 1\}\) is a subgroup, one can also consider the intermediate curve \(X_{H}(N)\) between \(X_{1}(N) \text{ and } X_{0}(N)\) defined as the quotient space \(H\bs X_{1}(N)\) where \((\Z/N\Z)^{\times}/\{\pm1\}\) acts as the diamond operators. Indeed, taking $H=1$ gives \(X_1(N)\) and taking \(H=(\Z/N\Z)^{\times}/\{\pm 1\}\) gives \(X_0(N)\).

To ease notation, we denote the Jacobians of the modular curves \(X_0(N)\text{ and }X_1(N)\) by \(J_0(N)\text{ and }J_1(N)\), respectively. These Jacobians are all defined over $\Q$.

\smallskip

\noindent\textbf{Rationality of Cusps.} We fix the level~\(N=p \ge 5\) prime. Let $\Gamma$ be a congruence subgroup containing $\Gamma_1(p)$ and consider the modular curve $X(\Gamma)$. Its cusps are defined over $\Q(\mu_p)$~\cite{shimura}*{Section~6.2} and can be identified with equivalence classes
\[\{(x,y)\in\Z^2: \gcd(x,y,p)=1\}\big/ \!\sim\] 
where $(x,y)\sim(x',y')$ if and only if there exists $\gamma=\smat{a}{b}{c}{d}\in\Gamma$ such that 
\begin{equation}    \label{eq:cusps}
(x',y')=\gamma(x,y)=(ax+by,cx+dy).
\end{equation}
We denote the class of $(x,y)$ by $[x:y]$. Let $\zeta_p\in\Gal(\Q(\mu_p)/\Q)=(\Z/p\Z)^{\times}$ be a generator of the Galois group. For each $d\in(\Z/p\Z)^{\times}$, let $\sigma_d$ represent the automorphism $\zeta_p\mapsto \zeta_p^d$. Then it acts on cusps 
\[\sigma_d([x:y]):=[x:dy]=\smat{1}{0}{0}{d}[x:y]=\diamondop{d}[x:y]\] 
as the usual diamond operators.~A similar statement applies to the cusps of~\(X_H(p)\) (cf.\ Theorem~1.3.1 in~\cite{stevens}).

The curve \(X_0(p)\) has two cusps $[1:0]$ and $[0:1]$, usually called \(\infty\) and \(0\), both of which are rational over \(\Q\). 

When we specialize to $\Gamma=\Gamma_1(p)$ in \eqref{eq:cusps}, we see that the cusps of $X_1(p)$ form orbit classes $\{\pm(x+by:y):b\in\Z\}$. One convenient choice of representatives is
\[
[x:p],\quad 1\le x\le\frac{p-1}{2},
\qquad
[1:y],\quad 1\le y\le\frac{p-1}{2}.
\]

Hence \(X_1(p)\) has \(p-1\) cusps, which split under the diamond operators into two orbits of size \((p-1)/2\).  The orbit containing \(\infty\) (the “\(\infty\)‑cusps”) consists of rational cusps; the other orbit (the “\(0\)‑cusps”) consists of cusps defined over \(\Q(\mu_{p})^+\), the maximal totally real subfield of the cyclotomic field. 

\section{Computing rational torsion on modular Jacobians}   \label{sec:strategy} 
In this section, we prove Theorem~\ref{theorem:CESholds}. As observed in~\eqref{eq:cusp-jac_inclusion} with \(\Gamma=\Gamma_1(p)\), the inclusion ``$\subseteq$'' follows from the theorems of Manin and Drinfeld. For the converse inclusion, our idea is to construct an upper bound for the rational torsion subgroup. To do this, we first recall some facts about the Hecke algebra.
\begin{subsection}{Hecke Operators on Modular Curves} \label{subsec:hecke_algebra}
Consider the modular curve $X_1(N)$ of level $N$. For each prime~$q$ not dividing~$N$, define the auxiliary modular curve
\[
  X_{1,0}(N,q):=\bigl(\Gamma_1(N)\cap\Gamma_0(q)\bigr)\bs\HH^*
\]
whose noncuspidal points over $\C$ classify triples $(E,P,C)$, where $E$ is a complex elliptic curve, $P$ is a point of order~$N$ and $C\subset E$ is a cyclic subgroup of order~$q$. There exist\footnote{Actually, degeneracy maps can be defined for a general modular curve of level $N$, but since we are interested in $J_1(N)$ we will restrict our attention to $X_1(N)$.} \emph{degeneracy maps}~\cite{diamond-im}*{II.7.3.}:
\begin{equation*}   
X_1(N)\xleftarrow{\alpha_q(N)} X_{1,0}(N,q) \xrightarrow{\beta_q(N)} X_1(N).
\end{equation*}
which are defined over $\Q$ and can be used to define correspondences on the modular curve $X_1(N)$, as described in~\cite{diamond-im}*{I.3.2.}. The term ``degeneracy'' comes from the fact that these maps, in a sense, ``forget'' part of the level structure. Indeed, note that their moduli-theoretic interpretation over the complex numbers is given by
\begin{equation*}    \label{eq:degmaps}
  \alpha_q(N)(E,P,C) = (E,P)
  \quad\text{and}\quad
  \beta_q(N)(E,P,C) = (E/C,P\bmod C).
\end{equation*}
Let $J_{1,0}(N,q):=J\bigl(X_{1,0}(N,q)\bigr)$. Using the Albanese and Picard functoriality, the degeneracy maps between modular curves induce maps
\begin{equation*}    \label{eq:degmaps_jac}
    \alpha_{q}(N)_*,\beta_q(N)_*\colon J_{1,0}(N,q)\rightarrow J_1(N),\quad \alpha_{q}(N)^*,\beta_q(N)^*\colon J_1(N)\rightarrow J_{1,0}(N,q).
\end{equation*}
on the corresponding Jacobians. For every prime $q$ coprime to $N$, we define the $q$-th \emph{Hecke operator} as
\begin{equation}    \label{eq:heckeop}
T_q=T_q(N):= \beta_q(N)_* \circ\alpha_q(N)^*\colon J_1(N)\to J_1(N).
\end{equation}
which is defined over $\Q$ (since $\alpha_q(N)$ and $\beta_q(N)$ are). The operators $T_q$ for a general prime $q$ are constructed in a similar way.

For each $d\in(\Z/N\Z)^{\times}$, the diamond operators $\diamondop{d}$ discussed before also naturally extend to endomorphisms of $J_1(N)$ which commute with the Hecke operator $T_q$. For simplicity, we denote them by the same symbol $\diamondop{d}$. 

We set $T_1:=1$, and having defined $T_q$ for $q$ prime, the Hecke operator $T_n$ for general $n$ can be characterized inductively by the relations:
\begin{align}   \label{align:hecke_rels}
T_{q^r} & = T_{q^{r-1}}T_q-\diamondop{q}qT_{q^{r-2}} & &\text{ if $r\geq 2$ and $q$ is coprime to $N$,}\\
T_{q^r} & = T_q^r & &\text{ if $q$ divides $N$,}\\
T_{mn} & = T_{m}T_{n} & &\text{ if $\gcd(m,n)=1$.}
\end{align}

\begin{remark} \label{remark:hecke_adjoint}
Some authors define $T_q$ as $\alpha_{*}\circ\beta^{*}$ instead of $\beta_{*}\circ\alpha^{*}$, Our choice was purely arbitrary, as there is a canonical equivalence and one can switch between the two definitions using the Atkin-Lehner involution~\cite{diamond-im}*{III. Remark~10.2.2}. 
\end{remark}

We write $\T(N)$ for the ring 
\[\Z[T_n:\forall n\geq1]=\Z[T_q,\diamondop{d}:q\text{ prime},d\in(\Z/N\Z)^{\times}].\]
This forms a $\Z$-algebra of endomorphisms of $J_1(N)$ defined over $\Q$, which is called \emph{Hecke algebra}.
\end{subsection}

The following two theorems are the main tools for establishing an upper bound for the rational torsion subgroup. The first theorem shows that we can construct this space using the annihilators of operators derived from the Hecke operators discussed earlier. The second theorem states that these annihilators essentially ``live'' in the cuspidal group, thus providing a link between the rational torsion subgroup and the rational cuspidal subgroup in the proof of the inclusion ``\(\supseteq\)'' in Theorem~\ref{theorem:CESholds}.

\begin{theorem} \label{Tq_ker}
    Let \(q\nmid p\) be a prime and \(P\in J_{1}(p)(\Q)_{\tors}\) such that \(q\) is odd or \(P\) has odd order. Then \((T_q-\diamondop{q}-q)(P)=0\).
\end{theorem}
\begin{proof}
    Let \(n\) be the order of \(P\). Then \((T_q-\diamondop{q}-q)(P)\in J_{1}(p)(\Q)\) is a point of order dividing \(n\). Recall that in \(\End(J_1(p)_{\F_q})\), we have the following Eichler-Shimura relations (see~\cite{diamond-im}*{p. 87})
    \begin{align} 
        T_{q,\F_q} &= \diamondop{q}\Frob_q+\Ver_q,   \label{Eichler-Shimura:1} \\    
         q         &= \Ver_q\circ \Frob_q,      \label{Eichler-Shimura:2}
    \end{align}
    where \(\Frob_q\) is the Frobenius on \(J_1(p)_{\F_q}\) and \(\Ver_q\) is its dual (Verschiebung). 
    If \(\Bar{P}\) is the reduction \( \bmod\, q \) of \(P\), then we have \(\Frob_q(\Bar{P})=\Bar{P}\). So, using relations \eqref{Eichler-Shimura:1} and \eqref{Eichler-Shimura:2}, we obtain
    \[
        T_{q,\F_q}(\Bar{P})=\diamondop{q}\Frob_q(\Bar{P})+\Ver(\Bar{P})=\diamondop{q}\Bar{P}+\Ver(\Frob_q(\Bar{P}))=\diamondop{q}\Bar{P}+q\Bar{P},
    \]
   giving \((T_{q,\F_q}-\diamondop{q}-q)(\bar{P})=0\). Since the reduction map is injective on \(J_{1}(p)(\Q)\) when \(q\) is odd, and it is injective on odd order torsion when \(q=2\), the claim follows.
\end{proof}

\begin{theorem}[Derickx]
    Let \(q\nmid N\) be a prime. We consider \(\eta_q:=T_q-\diamondop{q}-q\in\T\) as a correspondence on \(X_1(N)\), inducing an endomorphism of the divisor group of \(X_1(N)\) over \(\C\). Then the kernel of \(\eta_q\) is contained in the divisors supported in cusps.
\end{theorem}
\begin{proof}
    See~\cite{DKKS}*{Prop.~2.4}.
\end{proof}
When \(N = p\) is prime, we can describe the kernel of the operator \(\eta_q\) explicitly. On the one hand, by the above Theorem~\ref{Tq_ker}, \(\eta_q\) annihilates all divisors supported on rational cusps. On the other hand, using~\cite{parent00}*{Section 2.4} and the remark after Prop.~2.4 of \cite{DKKS}, one can see that \(\eta_q\) also annihilates divisors supported on irrational cusps if and only if they are invariant under the action of the diamond operator \(\diamondop{q}\).

\subsection{Computing with modular symbols} \label{subsection:modsym}
In order to make the theory discussed above practical for direct computations, we introduce modular symbols. They provide a finite presentation of the homology groups associated with modular curves, making them a powerful tool in arithmetic geometry. For more about modular symbols, see~\cite{stein}*{Ch.~8}.

For any distinct pair of cusps $\alpha,\beta\in \PP^1(\Q)$, we use $\{\alpha,\beta\}$ to denote an oriented path with endpoints $\alpha$ and $\beta$  that starts at $\alpha$ and ends at $\beta$; it corresponds to a directed loop on the Riemann surface $X(1)=\HH^*/\Gamma(1)$.

Let $\M_2$ be the free abelian group generated by the formal symbols $\{\alpha,\beta\}$, with $\alpha,\beta\in \PP^1(\Q)\}$, modulo the relations\footnote{The notation $\{\alpha,\beta\}$ is standard in the literature, but it can be confusing as suggesting that the order of $\alpha$ and $\beta$ is not important, whereas in fact it does matter. Indeed, note that $\{\alpha,\beta\} = -\{\beta,\alpha\}$ according to the relations.}
\[
\{ \alpha,\beta\} + \{\beta,\gamma\} + \{\gamma,\alpha\} = 0,
\]
and modulo any torsion. We equip $\M_2$ with a left $\SL_2(\Z)$-action by defining 
\[ g \{\alpha,\beta\} :=  \{g\alpha,g\beta\}.\]
We also define $\B_2$ as the free abelian group with basis $\PP^1(\Q)$, on which $\SL_2(\Z)$ acts on the left via linear fractional transformations.

For any congruence subgroup $\Gamma\subseteq \SL_2(\Z)$, we can view $\M_2$ and $\B_2$ as $\Z[\Gamma$]-modules. Let $I_\Gamma:=\langle g-1:g\in \Gamma\rangle\subseteq \Z[\Gamma]$ and consider the $\Gamma$-invariant quotients:
\begin{equation*}
(\M_2)^\Gamma:=\M_2/I_\Gamma \M_2,\qquad (\B_2)^\Gamma:=\B_2/I_\Gamma \B_2,
\end{equation*}

We define $\M_2(\Gamma)$ as the torsion-free quotient of $(\M_2)^\Gamma$ and we call it the space of (weight 2) \emph{modular symbols for $\Gamma$}. Similarly, let $\B_2(\Gamma)$ be the torsion-free quotient of $(\B_2)^\Gamma$; this is the space of (weight 2)~\emph{boundary symbols} for $\Gamma$. 

\begin{remark}
The space $\M_2(\Gamma)$ is finitely generated~\cite{stein}*{\textsection8, Prop.~1.24}.
\end{remark}

Let $X(\Gamma)$ be a modular curve and consider the homology $H_1(X(\Gamma),\{\text{cusps}\},\Z)$ \emph{relative to the cusps}; this means that in addition to loops, we include homology classes of paths that start and end at cusps. There is a natural homomorphism
\[
\varphi\colon \M_2(\Gamma)\to H_1(X(\Gamma),\{\text{cusps}\}, \Z)
\]
that sends $\Z$-linear combinations of geodesic paths in $\HH^*$ to their images in $X(\Gamma)=\Gamma\backslash \HH^*$.

We now define the \emph{boundary map}
\begin{align*}
\delta\colon \M_2(\Gamma) &\to \B_2(\Gamma)\\
\{\alpha,\beta\} &\mapsto \beta-\alpha.
\end{align*}
The kernel of this map is the subspace $\S_2(\Gamma)$ of (weight 2) \emph{cuspidal modular symbols} for $\Gamma$. Then the restriction of $\varphi$ to the subspace of cuspidal modular symbols defines an isomorphism of free $\Z$-modules (see~\cite{manin:parabolic}*{Thm.~1.9})
\begin{equation}    \label{equation:modsym_hom}
\S_2(\Gamma)\overset\sim\to H_1(X(\Gamma),\Z).
\end{equation} 

Moreover, modular symbols are equipped with analogues of the standard operators on spaces of modular forms which we will denote with the same symbols. To see how the degeneracy maps are defined on modular symbols see~\cite{stein}*{\textsection 8.6}, while for the diamond and Hecke operators see~\cite{stein}*{\textsection 8.3}. Here we only recall the star involution $\iota^*$ that takes care of complex conjugation: it is defined on a modular symbol $\{\alpha,\beta\}\in\M_2(\Gamma)$ as $\iota^*(\{\alpha,\beta\})=\{-\alpha,-\beta\}$ (note that this is the negative version of the one presented in Stein's book).

Let $\Omega^{1}_{\hol}(X(\Gamma))$ be the $\C$-vector space of holomorphic 1-forms on $X(\Gamma)$. There is a canonical pairing 
\begin{equation}  \label{equation:pairing}
\Omega^1_{\hol}(X(\Gamma)) \times \M_2(\Gamma)  \longrightarrow  \C\\
\end{equation}
defined by integrating \(\omega\) along the geodesic representing the modular symbol \(\{\alpha,\beta\}\) 
\[
  \langle\,\omega,\{\alpha,\beta\}\rangle
  \;=\;
  \int_{\alpha}^{\beta}\omega,
  \quad
  \omega\in\Omega^1_{\mathrm{hol}}(X(\Gamma)),\;
  \{\alpha,\beta\}\in\mathcal{M}_2(\Gamma),
\]

The pairing is compatible with respect to the action of Hecke operators on modular symbols and modular forms, i.e.\ for each prime \(\ell\),
\(\langle T_\ell \omega,\{\alpha,\beta\}\rangle
= \langle \omega, T_\ell \{\alpha,\beta\}\rangle\)
for all \(\{\alpha,\beta\}\in\M_2(\Gamma)\)~\cite{stein}*{Thm.~1.42}), and becomes perfect and nondegenerate when restricted to $\S_2(\Gamma)$~\cite{merel}*{Section~1.5}. Furthermore, from~\eqref{equation:pairing} we have the composition
\begin{equation} \label{equation:pair_embedding}
\pi \colon H_1(X(\Gamma), \text{cusps}, \Z) \longrightarrow (\Omega^1_{\hol}(X(\Gamma))^* \longrightarrow H_1(X(\Gamma), \mathbb{R})  
\end{equation}
which has image in $H_1(X(\Gamma),\Q)$ by the Manin--Drinfeld Theorem, so it gives an embedding of the relative integral homology into the rational homology group.

\subsection{Proof of Theorem~\ref{theorem:CESholds}}    \label{subsection:proof}
We verify Conjectures~\ref{conj:ces} and the equality in~\eqref{equation:diff_cusps-rat_cusps_inclusion} for the open cases $p=97,101,109$ and $113$. We proved this by carrying out explicit computations in Magma, using modular symbols. The raw data can be found in Section~\ref{section:tables}. Here we outline the strategy of the proof.

Let $X = X_1(p)$ be the modular curve of prime level $p\in\{97,101,109,113\}$, and set
\[
  H \;=\; H_1\bigl(X(\C),\Z\bigr), 
  \qquad
  V \;=\; H \otimes_\Z \Q \;=\; H_1\bigl(X(\C),\Q\bigr).
\] 
Then 
\[
  J_1(p)(\C)_{\tors}
  \;\cong\;
  V \big/ H
  \;=\; M(p),
\]
so that modular symbols give a concrete model for the torsion subgroup of the Jacobian.

By Theorem~\ref{Tq_ker}, the image of the rational torsion subgroup is killed by every operator
\[
  \eta_q \;=\; T_q \;-\; \langle q\rangle \;-\; q
  \quad
  (\text{for }q\neq p,\;q\text{ odd})
  \quad\text{and by}\quad
  \iota^* - 1,
\]
where $\iota^*$ is induced by complex conjugation.  We can explicitly compute Hecke and diamond endomorphisms on $V$ that induce $\eta_q$ on $M(p)$, along with the star involution $\iota^*$. In Magma we therefore compute the joint kernel
\[
  M'(p)
  \;:=\;
  \bigcap_{\substack{q\neq p\\q\text{ odd}}}
    \ker\bigl(\eta_q,\,\iota^*-1\bigr)
  \;\subseteq\;
  M(p),
\]
truncating to a finite set of small primes $q$ (see Table~\ref{subsec:tablelist2}).  One then checks
\[
  J_1(p)(\Q)_{\tors}
  \;\subseteq\;
  M'(p),
\]
giving an upper bound on the rational torsion (Table~\ref{subsec:tableproof}).

We can also compute the cuspidal subgroup $\cuspidal(p)$ as the image in $M(p)$ of the relative homology $H_1(X(\C),\text{cusps},\Z)$ via its embedding~\eqref{equation:pair_embedding} into $H_1(X(\C),\Q)$. We obtain that 
\[M'(p)\subseteq \cuspidal(p).\]

Finally, we use the description of the Galois action  on the \((p-1)\) cusps via diamond operators to compute the subgroup of $\cuspidal(p):=\cuspidal_1(p)$ invariant under \(\Gal(\bar{\Q}/\Q)\): the \(\tfrac{p-1}2\)  \(\infty\)-cusps remain fixed, while the $0$-cusps are cyclically permuted. It follows that
\[J_1(p)(\Q)_{\tors} =  \cuspidal(p)^{\Gal_{\Q}}.\]

An independent Magma check shows that the subgroup \(\cuspidal(p)^{\Q}\) generated by differences of rational cusps  coincides with \(\cuspidal(p)^{\Gal_{\Q}}\) (compare the tables in~\ref{subsec:tableproof}). This completes the proof. \qed
\input{booktables}


\end{document}

%% file: booktables.tex
\section{Tables}\label{section:tables}
\subsection{Tables for the proof of Theorem \ref{theorem:CESholds}}\label{subsec:tableproof}

\noindent The tables below contain the invariants of the bounding group $M'(p)$, the cuspidal group \( \cuspidal(p) \), the rational cuspidal group \(\cuspidal(p)^{\Gal_{\Q}} \) and the subgroup $\cuspidal(p)^{\Q}$ generated by differences of rational cusps computed with Magma. Observe that $\cuspidal(p)^{\Gal_{\Q}}  =\cuspidal(p)^{\Q}$ as we expected.  By \emph{invariants} we mean the elementary divisors in the decomposition of a finite abelian group into a direct product of cyclic subgroups.  In the tables below, an entry \[G=\bigl[2 \mid 2^2 \mid 2^2 \cd 5^2 \cd 7 \bigr], \text{ denotes that }G\cong\frac{\Z}{2\Z} \times \frac{\Z}{2^2 \Z} \times \frac{\Z}{2^2 \cd 5^2 \cd 7 \Z}.\]  

\begin{longtable}{c p{10cm}}
  \caption{Table for \(p = 97\)}\label{tab:p97}\\
  \toprule
  \textbf{Groups} & \textbf{Invariants} \\
  \midrule
  \endfirsthead

  \multicolumn{2}{@{}l}{\small\slshape Continued from previous page}\\
  \midrule
  \textbf{Groups} & \textbf{Invariants} \\
  \midrule
  \endhead

  \midrule
  \multicolumn{2}{@{}r}{\small\slshape Continued on next page}\\
  \endfoot

  \bottomrule
  \endlastfoot

  \(M'(97)\)
    & \(\bigl[2\cdot5\cdot7 \mid 2^{4}\cdot5\cdot7\cdot17\cdot149\cdot241\cdot367\cdot421\cdot2753\cdot147689\cdot651997\cdot21205889\cdot41481169\cdot5429704177\cdot2758053952369\bigr]\) \\
  \midrule
  \(\cuspidal(97)\)
    & \(\bigl[5\cdot7 \mid 5\cdot7 \mid 2^{2}\cdot5\cdot7\cdot17\cdot149\cdot241\cdot367\cdot421\cdot2753\cdot147689\cdot651997\cdot21205889\cdot41481169\cdot5429704177\cdot2758053952369 \mid 2^{6}\cdot5\cdot7\cdot17\cdot149\cdot241\cdot367\cdot421\cdot2753\cdot147689\cdot651997\cdot21205889\cdot41481169\cdot5429704177\cdot2758053952369\bigr]\) \\
  \midrule
  \(\cuspidal(97)^{\Q}\)
    & \(\bigl[5\cdot7 \mid 2^{4}\cdot5\cdot7\cdot17\cdot149\cdot241\cdot367\cdot421\cdot2753\cdot147689\cdot651997\cdot21205889\cdot41481169\cdot5429704177\cdot2758053952369\bigr]\) \\
  \midrule
  \(\cuspidal(97)^{\Gal_{\Q}}\)
    & \(\bigl[5\cdot7 \mid 2^{4}\cdot5\cdot7\cdot17\cdot149\cdot241\cdot367\cdot421\cdot2753\cdot147689\cdot651997\cdot21205889\cdot41481169\cdot5429704177\cdot2758053952369\bigr]\) \\

  \midrule
  Total execution time & \(93705.919\,\mathrm{s}\approx26\,\mathrm{h}\)
\end{longtable}

\begin{longtable}{c p{10cm}}
  \caption{Table for \(p = 101\)}\label{tab:p101}\\
  \toprule
  \textbf{Groups} & \textbf{Invariants} \\
  \midrule
  \endfirsthead

  \multicolumn{2}{@{}l}{\small\slshape Continued from previous page}\\
  \midrule
  \textbf{Groups} & \textbf{Invariants} \\
  \midrule
  \endhead

  \midrule
  \multicolumn{2}{@{}r}{\small\slshape Continued on next page}\\
  \endfoot

  \bottomrule
  \endlastfoot

  \(M'(101)\)
    & \(\bigl[5^{2}\cdot19\cdot101\cdot1201\cdot52951\cdot54371\cdot599491\cdot1493651\cdot12355051\cdot709068505801\cdot58884077243434864347851\bigr]\) \\
  \midrule
  \(\cuspidal(101)\)
    & \(\bigl[19\cdot101\cdot1201\cdot52951\cdot54371\cdot599491\cdot1493651\cdot12355051\cdot709068505801\cdot58884077243434864347851 \mid 5^{4}\cdot19\cdot101\cdot1201\cdot52951\cdot54371\cdot599491\cdot1493651\cdot12355051\cdot709068505801\cdot58884077243434864347851\bigr]\) \\
  \midrule
  \(\cuspidal(101)^{\Q}\)
    & \(\bigl[5^{2}\cdot19\cdot101\cdot1201\cdot52951\cdot54371\cdot599491\cdot1493651\cdot12355051\cdot709068505801\cdot58884077243434864347851\bigr]\) \\
  \midrule
  \(\cuspidal(101)^{\Gal_{\Q}}\)
    & \(\bigl[5^{2}\cdot19\cdot101\cdot1201\cdot52951\cdot54371\cdot599491\cdot1493651\cdot12355051\cdot709068505801\cdot58884077243434864347851\bigr]\) \\

  \midrule
  Total execution time & \(279276.410\,\mathrm{s}\approx77\text{--}78\,\mathrm{h}\)
\end{longtable}

\begin{longtable}{c p{10cm}}
  \caption{Table for \(p = 109\)}\label{tab:p109}\\
  \toprule
  \textbf{Groups} & \textbf{Invariants} \\
  \midrule
  \endfirsthead

  \multicolumn{2}{@{}l}{\small\slshape Continued from previous page}\\
  \midrule
  \textbf{Groups} & \textbf{Invariants} \\
  \midrule
  \endhead

  \midrule
  \multicolumn{2}{@{}r}{\small\slshape Continued on next page}\\
  \endfoot

  \bottomrule
  \endlastfoot

  \(M'(109)\)
    & \(\bigl[2 \mid 2\cdot3^{3} \mid 2^{2}\cdot3^{3}\cdot37\cdot127 \mid 2^{2}\cdot3^{6}\cdot37\cdot103\cdot127\cdot3187\cdot22483\cdot129763\cdot2230759\cdot144218626120352809\cdot7225241488211218811391927451\bigr]\) \\
  \midrule
  \(\cuspidal(109)\)
    & \(\bigl[2^{2}\cdot3\cdot37\cdot127 \mid 2^{2}\cdot3^{5}\cdot37\cdot127 \mid 2^{2}\cdot3^{6}\cdot37\cdot103\cdot127\cdot3187\cdot22483\cdot129763\cdot2230759\cdot144218626120352809\cdot7225241488211218811391927451 \ |\ 2^{2}\cdot3^{6}\cdot37\cdot103\cdot127\cdot3187\cdot22483\cdot129763\cdot2230759\cdot144218626120352809\cdot7225241488211218811391927451\bigr]\) \\
  \midrule
  \(\cuspidal(109)^{\Q}\)
    & \(\bigl[2^{2}\cdot3^{3}\cdot37\cdot127 \mid 2^{2}\cdot3^{6}\cdot37\cdot103\cdot127\cdot3187\cdot22483\cdot129763\cdot2230759\cdot144218626120352809\cdot7225241488211218811391927451\bigr]\) \\
  \midrule
  \(\cuspidal(109)^{\Gal_{\Q}}\)
    & \(\bigl[2^{2}\cdot3^{3}\cdot37\cdot127 \mid 2^{2}\cdot3^{6}\cdot37\cdot103\cdot127\cdot3187\cdot22483\cdot129763\cdot2230759\cdot144218626120352809\cdot7225241488211218811391927451\bigr]\) \\

  \midrule
  Total execution time & \(269108.450\,\mathrm{s}\approx74\text{--}75\,\mathrm{h}\)
\end{longtable}

\begin{longtable}{c p{10cm}}
  \caption{Table for \(p = 113\)}\label{tab:p113}\\
  \toprule
  \textbf{Groups} & \textbf{Invariants} \\
  \midrule
  \endfirsthead

  \multicolumn{2}{@{}l}{\small\slshape Continued from previous page}\\
  \midrule
  \textbf{Groups} & \textbf{Invariants} \\
  \midrule
  \endhead

  \midrule
  \multicolumn{2}{@{}r}{\small\slshape Continued on next page}\\
  \endfoot

  \bottomrule
  \endlastfoot

  \(M'(113)\)
    & \(\bigl[2 \mid 2 \mid 2 \mid 2 \mid 2 \mid 2 \mid 2^{2} \mid 2^{2} \mid 2^{4} \mid 2^{4} \mid 2^{4}\cdot13 \mid 2^{4}\cdot3^{2}\cdot5\cdot7\cdot13\cdot41\cdot1597\cdot2689\cdot5419\cdot7393\cdot33181\cdot47609\cdot83685281\cdot1338273009109\cdot3747533743340403014797054313\bigr]\) \\
  \midrule
  \(\cuspidal(113)\)
    & \(\bigl[2^{2} \mid 2^{2} \mid 2^{2} \mid 2^{2} \mid 2^{4} \mid 2^{4} \mid 2^{4} \mid 2^{4} \mid 2^{4}\cdot13 \mid 2^{4}\cdot13 \mid 2^{4}\cdot3^{2}\cdot5\cdot13\cdot41\cdot1597\cdot2689\cdot5419\cdot7393\cdot33181\cdot47609\cdot83685281\cdot1338273009109\cdot3747533743340403014797054313 \mid 2^{4}\cdot3^{2}\cdot5\cdot7^{2}\cdot13\cdot41\cdot1597\cdot2689\cdot5419\cdot7393\cdot33181\cdot47609\cdot83685281\cdot1338273009109\cdot3747533743340403014797054313\bigr]\) \\
  \midrule
  \(\cuspidal(113)^{\Q}\)
    & \(\bigl[2^{2} \mid 2^{2} \mid 2^{4} \mid 2^{4} \mid 2^{4}\cdot13 \mid 2^{4}\cdot3^{2}\cdot5\cdot7\cdot13\cdot41\cdot1597\cdot2689\cdot5419\cdot7393\cdot33181\cdot47609\cdot83685281\cdot1338273009109\cdot3747533743340403014797054313\bigr]\) \\
  \midrule
  \(\cuspidal(113)^{\Gal_{\Q}}\)
    & \(\bigl[2^{2} \mid 2^{2} \mid 2^{4} \mid 2^{4} \mid 2^{4}\cdot13 \mid 2^{4}\cdot3^{2}\cdot5\cdot7\cdot13\cdot41\cdot1597\cdot2689\cdot5419\cdot7393\cdot33181\cdot47609\cdot83685281\cdot1338273009109\cdot3747533743340403014797054313\bigr]\) \\

  \midrule
  Total execution time & \(382415.860\,\mathrm{s}\approx106\,\mathrm{h}\)
\end{longtable}

\subsection{Table with the rational torsion subgroups of \texorpdfstring{$J_1(p)$}{J1(N)}}  \label{subsec:tablelist1}
In the following table we provide the explicit structure of the rational torsion subgroup of $J_1(p)$ for $5\le p\le 113$, found with the same method as described in Section~\ref{subsection:proof}.

\begin{longtable}{c p{12cm}}
  \caption{Rational torsion subgroup of \(J_1(p)\)}\label{tab:J1tors}\\
  \toprule
  \textbf{Prime \(p\)} & \textbf{\(J_1(p)(\Q)_{\tors}\)} \\
  \midrule
  \endfirsthead

  \multicolumn{2}{@{}l}{\small\slshape Continued from previous page}\\
  \midrule
  \textbf{Prime \(p\)} & \textbf{\(J_1(p)(\Q)_{\tors}\)} \\
  \midrule
  \endhead

  \midrule
  \multicolumn{2}{@{}r}{\small\slshape Continued on next page}\\
  \endfoot

  \bottomrule
  \endlastfoot

  5   & \(\mathrm{trivial}\) \\
  \midrule
  7   & \(\mathrm{trivial}\) \\
  \midrule
  11  & \(\bigl[5\bigr]\) \\
  \midrule
  13  & \(\bigl[19\bigr]\) \\
  \midrule
  17  & \(\bigl[2^{3}\cdot73\bigr]\) \\
  \midrule
  19  & \(\bigl[3^{2}\cdot487\bigr]\) \\
  \midrule
  23  & \(\bigl[11\cdot37181\bigr]\) \\
  \midrule
  29  & \(\bigl[2^{2} \mid 2^{2} \mid 2^{2} \cdot 3\cdot7\cdot43\cdot17837\bigr]\) \\
  \midrule
  31  & \(\bigl[2\cdot5 \mid 2\cdot5\cdot7\cdot11\cdot2302381\bigr]\) \\
  \midrule
  37  & \(\bigl[3^{2}\cdot5\cdot7\cdot19\cdot37\cdot73\cdot577\cdot17209\bigr]\) \\
  \midrule
  41  & \(\bigl[2^{4}\cdot5\cdot13\cdot31^{2}\cdot431\cdot250183721\bigr]\) \\
  \midrule
  43  & \(\bigl[2 \mid 2\cdot7\cdot19\cdot29\cdot463\cdot1051\cdot416532733\bigr]\) \\
  \midrule
  47  & \(\bigl[23\cdot139\cdot82397087\cdot12451196833\bigr]\) \\
  \midrule
  53  & \(\bigl[7\cdot13\cdot85411\cdot96331\cdot379549\cdot641949283\bigr]\) \\
  \midrule
  59  & \(\bigl[29\cdot59\cdot9988553613691393812358794271\bigr]\) \\
  \midrule
  61  & \(\bigl[7\cdot11 \mid 5\cdot7\cdot11\cdot19\cdot31\cdot2081\cdot2801\cdot40231\cdot411241\cdot514216621\bigr]\) \\
  \midrule 
  67  & \(\bigl[661 \mid 11\cdot67\cdot193\cdot661\cdot2861\cdot8009\cdot11287\cdot9383200455691459\bigr]\) \\
  \midrule
  71  & \(\bigl[701 \mid 5\cdot7\cdot31\cdot113\cdot211\cdot281\cdot701\cdot12713\cdot13070849919225655729061\bigr]\) \\
  \midrule
  73  & \(\bigl[2 \mid 2 \mid 2\cdot3^{2}\cdot11\cdot79\cdot89\cdot241\cdot23917\cdot3341773\cdot11596933\cdot31964959893317833\bigr]\) \\
  \midrule
  79  & \(\bigl[521 \mid 13\cdot157\cdot199\cdot521\cdot1249\cdot4447\cdot323623\cdot1130429\cdot68438648614508149381\bigr]\) \\
  \midrule
  83  & \(\bigl[41\cdot17210653\cdot151251379\cdot18934761332741\cdot48833370476331324749419\bigr]\) \\
  \midrule
  89  & \(\bigl[2 \mid 2 \mid 2\cdot5\cdot11\cdot13\cdot37\cdot397\cdot4027\cdot262504573\cdot15354699728897\cdot49135060828995551670374357\bigr]\) \\
  \midrule
  97  & \(\bigl[5\cdot7 \mid 2^{4}\cdot5\cdot7\cdot17\cdot149\cdot241\cdot367\cdot421\cdot2753\cdot147689\cdot651997\cdot21205889\cdot41481169\cdot5429704177\cdot2758053952369\bigr]\) \\
  \midrule
  101 & \(\bigl[5^{2}\cdot19\cdot101\cdot1201\cdot52951\cdot54371\cdot599491\cdot1493651\cdot12355051\cdot709068505801\cdot58884077243434864347851\bigr]\) \\
  \midrule
  103 & \(\bigl[17 \mid 7^{2}\cdot13\cdot17\cdot103\cdot613\cdot100458793666879\cdot123953701101455911613\cdot60417254667158883466061055469\bigr]\) \\
  \midrule
  107 & \(\bigl[53\cdot304009\cdot1598587\cdot7762787405087851\cdot1827219997313025527\cdot340411510885100431606787699221\bigr]\) \\
  \midrule
  109 & \(\bigl[2^{2}\cdot3^{3}\cdot37\cdot127 \mid 2^{2}\cdot3^{6}\cdot37\cdot103\cdot127\cdot3187\cdot22483\cdot129763\cdot2230759\cdot144218626120352809\cdot7225241488211218811391927451\bigr]\) \\
  \midrule
  113 & \(\bigl[2^{2} \mid 2^{2} \mid 2^{4} \mid 2^{4} \mid 2^{4}\cdot13 \mid 2^{4}\cdot3^{2}\cdot5\cdot7\cdot13\cdot41\cdot1597\cdot2689\cdot5419\cdot7393\cdot33181\cdot47609\cdot83685281\cdot1338273009109\cdot3747533743340403014797054313\bigr]\) \\
\end{longtable}

\begin{longtable}{ c  c  c  c }
  \caption{Here, \(T_q\) denotes the Hecke operator used to obtain the upper bound, \(\langle d\rangle\) is the diamond operator describing the Galois action, and the final column records the program’s execution time.}\label{subsec:tablelist2}\\
  \toprule
  \textbf{Prime \(p\)} & \textbf{Hecke Operator \(T_q\)} & \textbf{Galois Action \(\langle d\rangle\)} & \textbf{Time (s)} \\
  \midrule
  \endfirsthead

  \multicolumn{4}{@{}l}{\small\slshape Continued from previous page}\\
  \midrule
  \textbf{Prime \(p\)} & \textbf{Hecke Operator \(T_q\)} & \textbf{Galois Action \(\langle d\rangle\)} & \textbf{Time (s)} \\
  \midrule
  \endhead

  \midrule
  \multicolumn{4}{@{}r}{\small\slshape Continued on next page}\\
  \endfoot

  \bottomrule
  \endlastfoot

  11  & \(T_{3}\)   & \(\langle2\rangle\) & 0.160    \\
  \midrule
  13  & \(T_{3}\)   & \(\langle2\rangle\) & 0.210    \\
  \midrule
  17  & \(T_{3}\)   & \(\langle3\rangle\) & 0.660    \\
  \midrule
  19  & \(T_{3}\)   & \(\langle2\rangle\) & 1.190    \\
  \midrule
  23  & \(T_{3}\)   & \(\langle5\rangle\) & 3.230    \\
  \midrule
  29  & \(T_{3}\)   & \(\langle2\rangle\) & 16.170   \\
  \midrule
  31  & \(T_{3}\)   & \(\langle3\rangle\) & 14.710   \\
  \midrule
  37  & \(T_{3}\)   & \(\langle2\rangle\) & 51.560   \\
  \midrule
  41  & \(T_{3}\)   & \(\langle6\rangle\) & 109.260  \\
  \midrule
  43  & \(T_{11}\)  & \(\langle3\rangle\) & 217.510  \\
  \midrule
  47  & \(T_{3}\)   & \(\langle5\rangle\) & 287.510  \\
  \midrule
  53  & \(T_{3}\)   & \(\langle2\rangle\) & 705.280  \\
  \midrule
  59  & \(T_{3}\)   & \(\langle2\rangle\) & 1585.410 \\
  \midrule
  61  & \(T_{11}\)  & \(\langle2\rangle\) & 3226.790 \\
  \midrule
  67  & \(T_{17}\)  & \(\langle2\rangle\) & 16145.930\\
  \midrule
  71  & \(T_{3}\)   & \(\langle7\rangle\) & 5986.440 \\
  \midrule
  73  & \(T_{3}\)   & \(\langle5\rangle\) & 8313.790 \\
  \midrule
  79  & \(T_{3}\)   & \(\langle3\rangle\) & 15867.310\\
  \midrule
  83  & \(T_{3}\)   & \(\langle2\rangle\) & 24183.350\\
  \midrule
  89  & \(T_{5}\)   & \(\langle3\rangle\) & 46008.870\\
  \midrule
  97  & \(T_{3}\)   & \(\langle5\rangle\) & 93705.620\\
  \midrule
  101 & \(T_{11}\)  & \(\langle2\rangle\) & 279276.410\\
  \midrule
  103 & \(T_{3}\)   & \(\langle5\rangle\) & 156711.890\\
  \midrule
  107 & \(T_{3}\)   & \(\langle2\rangle\) & 219285.870\\
  \midrule
  109 & \(T_{3}\)   & \(\langle6\rangle\) & 269108.450\\
  \midrule
  113 & \(T_{5}\)   & \(\langle3\rangle\) & 382415.860\\

\end{longtable}